\documentclass[10pt]{article}

\usepackage[british]{babel}
\usepackage{latexsym}
\usepackage{amsmath}
\usepackage{amsthm}
\usepackage{amssymb}
\usepackage{enumerate}
\usepackage{url}
\usepackage[a4paper,textwidth=6in,textheight=9in]{geometry}

\newtheorem{lemma}{Lemma}[section]
\newtheorem{theorem}[lemma]{Theorem}
\newtheorem{corollary}[lemma]{Corollary}

\title{Reconstructing the topology on monoids and polymorphism clones of the rationals}
\author{Mike Behrisch\thanks{%
The research of the first named author was partly supported by the Austrian
Science Fund (FWF) under grant no.~I836-N23.},
Tech\-ni\-sche Uni\-ver\-si\-t\"{a}t Wien\thanks{%
Institute for Discrete Mathematics and Geometry, Wied\-ner Haupt\-stra{\ss}e~8--10, 1040~Wien, Austria}, and\and
John K Truss and Edith Vargas-Garc\'ia\thanks{%
The research of the third author was supported by CONACYT.},
University of Leeds\thanks{%
Department of Pure Mathematics, University of Leeds, Leeds LS2 9JT, UK}.%
\thanks{We thank Christian and Maja Pech for helpful remarks.}}
\date{e-mails: \url{behrisch@logic.at},  \url{pmtjkt@leeds.ac.uk}, \url{alemaniamir5@gmail.com}.}%

\begin{document}
\maketitle

\begin{abstract} We show how to reconstruct the topology on the monoid of endomorphisms of the rational numbers under the
strict or reflexive order relation, and the polymorphism clone of the rational numbers under the reflexive relation.
In addition we show how automatic homeomorphicity results can be lifted to polymorphism clones generated by monoids.
\end{abstract}

2010 Mathematics Subject Classification 08A35

keywords: rationals, automatic homeomorphicity, embedding, endomorphism, polymorphism clone

\section{Introduction}
\label{sect:1}
We write~$M$ and~$E$ for the monoids of endomorphisms of
$(\mathbb{Q}, <)$ (coinciding with the self-embeddings of
$(\mathbb{Q},<)$ since~$<$ is linear) and $(\mathbb{Q}, \le)$ respectively, and~$G$ for
the automorphism group $\mathop{\mathrm{Aut}}(\mathbb{Q}, <)$ (which equals $\mathop{\mathrm{Aut}}(\mathbb{Q}, \le)$), so that~$G$ is the
family of invertible members of~$E$. Our main results are that~$M$ and~$E$ have
`automatic homeomorphicity' in the sense of~\cite{Bodirsky}, with a corresponding result for the polymorphism clone of $(\mathbb{Q}, \le)$. This intuitively means that
the natural topology (see below for a precise definition) can be recognized inside the algebraic structure; the more formal
definition says that any isomorphism from~$M$ to a closed submonoid of the full transformation monoid on a countable set is
also necessarily a homeomorphism, with an analogous statement for~$E$ and $\mathop{\mathrm{Pol}}(\mathbb{Q}, \le)$.
The case of $\mathop{\mathrm{Pol}}(\mathbb{Q}, <)$ is not yet solved. According to the treatment given in~\cite{Bodirsky}, the main preliminary
technical result needed to demonstrate automatic homeomorphicity for~$M$ is that any injective endomorphism of~$M$ which fixes~$G$ pointwise, also
fixes every member of~$M$ (since~$M$ is the closure of~$G$). We need a
slightly more general variant of this fact for the proofs regarding the monoid~$E$.
We are also able to show the truth of the corresponding statement
for automorphisms of~$E$ (though the deduction of automatic homeomorphicity requires more work, and a new and direct method,
since~$G$ is not dense in~$E$). We may identify~$M$ as the family of injective members of~$E$; another important monoid which
plays a role in some of the proofs is that of the surjective maps in~$E$,
denoted by~$S$, which we show to coincide with the epimorphisms of
$(\mathbb{Q}, \le)$. We shall see that, in fact, each
of these is definable (in the monoid language) in~$E$.

In order to establish the mentioned preliminary result for~$M$, we somehow have to represent the members of~$M$ inside~$G$. The most natural
and obvious way to attempt to do this is via centralizers. Indeed in similar `interpretability' results for~$G$, this is often
sufficient, and in our case, we can make quite good progress using this idea. To be more concrete, let us write~$\xi$ for the
given injective endomorphism of~$M$ which fixes~$G$ pointwise. Given any $f \in M$, it is natural to consider
$C_G(f) = \{g \in G\colon fg = gf\}$. For certain values of~$f$, we can show that $C_G(f) = C_G(f_1) \Rightarrow f = f_1$. From
this it easily follows that $\xi(f) = f$. This is because the centralizers of~$f$ and $\xi(f)$ are equal, as
$g \in C_G(\xi(f)) \Leftrightarrow g\xi(f) = \xi(f)g \Leftrightarrow \xi(gf) = \xi(fg)$ (as~$\xi$ fixes members of~$G$)
$\Leftrightarrow gf = fg \Leftrightarrow g \in C_G(f)$, and so from $C_G(f) = C_G(\xi(f))$ we deduce that $\xi(f) = f$. Among
elements~$f$ to which this applies are those of the form $f(x) = x$ if $x < \pi$, $x + 1$ if $x > \pi$ (and similarly for any
other irrational), as well as many others. An example of an~$f$ to which this does \emph{not} apply is $f(x) = x$ if $x < 0$,
$x + 1$ if $x \ge 0$ (which shares a centralizer with~$f_1$ given by $f_1(x) = x$ if $x \le 0$, $x + 1$ if $x > 0$). In the
case of order-preserving permutation groups, arguments using centralizers
are widespread, and solve many problems. See~\cite{Glass} for material on this.

To prove that $\xi(f) = f$ for general~$f$ is however more involved, and a technique described in~\cite{Bodirsky} which uses
sets of \emph{pairs} of group elements rather than subsets of~$G$ is used instead. This is
$S(f) = \{(\alpha, \beta) \in G^2\colon \alpha f = f \beta\}$. Our method is then to find certain subfamilies of~$M$, which we
denote by~$\Gamma$, $\Gamma^+$, $\Gamma^-$, and $\Gamma^\pm$, and show that for
$f \in \Gamma \cup \Gamma^+ \cup \Gamma^- \cup \Gamma^\pm$, $S(f) = S(f_1) \Leftrightarrow f = f_1$, from which by essentially
the same proof as above, $\xi(f) = f$. Then we show that for any member~$f$ of~$M$ there are $g_1, g_2$ lying in one of
$\Gamma$, $\Gamma^+$, $\Gamma^-$, $\Gamma^\pm$ such that $g_1f = g_2$, and use a trick involving cancellation  to conclude
the proof. We need a few technicalities to
achieve this. From this it immediately follows by results from~\cite{Bodirsky} that~$M$ has automatic
homeomorphicity (Theorem~\ref{2.6}).

In the next section we move on to a discussion of the endomorphism monoid~$E$ of $(\mathbb{Q}, \le)$. We can show that
various natural subsets of~$E$ are definable in~$E$, and we can lift the technical result concerning the map~$\xi$ to this
context too (assuming that it is an automorphism).

In section~\ref{sect:4}, we discuss the analogous result for~$E$. The key idea here is to analyze directly the possible actions of~$E$ on
a countable set~$\Omega$, which is the set which features in the definition
of `automatic homeomorphicity'. Any isomorphism of~$E$ to a closed
submonoid of the transformation monoid on~$\Omega$ gives rise to a monoid `action' of~$E$ on~$\Omega$. We
focus on the \emph{group} orbits of $G \subseteq E$, and use them to guide our analysis. Provided we know that the restriction
of the isomorphism to~$M$ maps it to a \emph{closed} submonoid, we can deduce from Theorem~\ref{2.6} that it is a
homeomorphism. Closedness of this image is easy to prove if the
isomorphism sends constants to constants, but examples show that, in
general, we cannot rely on this property. Thus, we have to invent another
method to ensure closedness, which is done by generalizing Lemma~12
from~\cite{Bodirsky}.
After this is solved, we are able to demonstrate precisely how the members of~$M$ act on~$\Omega$, and using the
technical lemmas from section~\ref{sect:3}, we can then \emph{directly} describe how~$E$ acts, and show that the isomorphism assumed to
exist must also be a homeomorphism.

In section~\ref{sect:5}, we use the earlier results to lift automatic homeomorphicity to the corresponding polymorphism clone. So
far this argument only works in the reflexive case, since `idempotents'
with \emph{finite} image are required in the proof, which exist in~$E$ but not
in~$M$. The problem highlighted in the previous paragraph for the monoid does not cause difficulties here however, since
the fact that for clones the images of constants are necessarily constants avoids the difficulty.

In the final section, we give a method for lifting automatic
homeomorphicity (and also automatic continuity---a variant, where every
homomorphism into the full transformation monoid / clone on a countable set is required to be continuous) from monoids to the clones they generate.
In this context, we can immediately deduce that the polymorphism clones $\langle \mathop{\mathrm{End}}(\mathbb{Q}, <) \rangle$
and $\langle \mathop{\mathrm{End}}(\mathbb{Q}, \leq) \rangle$ generated by $\mathop{\mathrm{End}}(\mathbb{Q}, <)$
and $\mathop{\mathrm{End}}(\mathbb{Q}, \leq)$, respectively, have automatic
homeomorphicity. These clones are (rather small) subclones of the corresponding full polymorphism clones.

To make sense of results about continuity, we need to recall what the topology is, on~$G$, $M$, $E$, and indeed also on the
clone. For~$E$, we take as sub-basic open sets all sets of the form
$\mathcal{B}_{qr} = \mbox{$\{f \in E\colon f(q) = r\}$}$, and the
topologies on~$G$ and~$M$ are then the induced ones. Basic open sets are then finite intersections of these, so have the form
$\{f \in E\colon \mathord{f\restriction_B} = g\}$, where~$B$ is a finite subset
of~$\mathbb Q$ and $g\colon B\to\mathbb{Q}$. In the polymorphism clone~$P$, the same sets are
used, but with higher `arities'. Thus for each~$n$, and $q_1, q_2, \ldots, q_n, r \in \mathbb{Q}$,
$\mathcal{B}_{q_1q_2\ldots q_nr} = \{f \in P^{(n)}\colon f(q_1, q_2, \ldots, q_n) = r\}$ is taken as a sub-basic open set. Similarly, in the
full transformation monoid $\mathop{\mathrm{Tr}}(\Omega)$ on a set~$\Omega$ (usually countable), sub-basic open sets have the form
$\{f \in \mathop{\mathrm{Tr}}(\Omega)\colon f(x) = y\}$ for $x, y \in \Omega$, with basic open sets as finite intersections of these, and
sub-basic open sets on the polymorphism clone similarly given by allowing increased arities. We remark that saying that~$G$ is
dense in~$M$ thus says that any embedding of $(\mathbb{Q}, <)$ can be approximated by automorphisms on arbitrarily large
finite sets. Therefore, saying that~$M$ is the closure of~$G$ says that any limit of members of~$G$ lies in~$M$ and any member
of~$M$ may be expressed as such a limit.
It is worth noting that, since~$G$, $M$ and~$E$ or the corresponding monoids
on~$\Omega$ live on a countable carrier set, their topology is actually
metrizable by an ultrametric (see~\cite[1.1, p.~132]{Pech} for details). This enables us to use sequential convergence
and continuity instead of the net analogues needed in general, and we shall
exploit this, for instance, in section~\ref{sect:6}.

An alternative proof of the main technical result Corollary~\ref{2.5} was given independently by James Hyde~\cite{Hyde} (not
using methods from~\cite{Bodirsky}).

\section{Main technical lemmas for~$M$}
\label{sect:2}

Throughout this section we suppose that~$\xi$ is an injective endomorphism of~$M$ which fixes~$G$ pointwise, and our goal is
to show that it also fixes~$M$ pointwise. It is fairly easy to show by `bare hands' that there are some members of~$M$ which
must be fixed, for instance those~$f$ that are characterized by their centralizers (meaning that if~$f$ and~$f'$ have equal
centralizers in~$G$, then they are equal). However, this type of argument only applies to a limited range of members of~$M$,
and we need a more systematic approach. For this we isolate particular subfamilies of members of~$M$, written~$\Gamma$,
$\Gamma^+$, $\Gamma^-$, $\Gamma^\pm$, show that all their members are fixed, and then lift this to all members~$f$ of~$M$
by writing~$f$ in terms of members of $\Gamma \cup \Gamma^+ \cup \Gamma^- \cup \Gamma^\pm$. To describe what~$\Gamma$ is, we
require the following definition.

The \emph{$2$-coloured version of the rationals} denoted by $\mathbb{Q}_2$, can be characterized as the set~$\mathbb Q$ of
rational numbers, together with a colouring function $F\colon \mathbb{Q} \to C = \{\mbox{red, blue}\}$ such that for every
$x < y$ in~$\mathbb Q$ and $c \in C$ there is $z\in \mathbb{Q}$ with $x < z < y$ and $F(z) = c$. It is well known that this
exists and is unique up to isomorphism.

A key observation in what we do is that the only relevant information about $f \in M$ for our present purposes is the value of
its \emph{image}. This is because, if~$f_1$ and~$f_2$ in~$M$ have the same image, then $f_2^{-1}f_1$ is (defined and) an
automorphism, so by hypothesis is fixed by~$\xi$, and, since
$f_2(f_2^{-1}f_1) = f_1$, it is immediate that~$f_1$ is fixed if and only if~$f_2$ is. So we
really need to focus mainly on subsets of~$\mathbb Q$, though we often construe them as images. In fact with regard to this,
it is clear that a subset of~$\mathbb Q$ is the image of some self-embedding if and only if it is isomorphic to~$\mathbb Q$.

With this in mind, for any $A \subseteq \mathbb{Q}$ isomorphic to~$\mathbb Q$, let us define a relation~$\sim$ on~$\mathbb Q$
by $x \sim y$ if there is at most one point of~$A$ strictly between~$x$ and~$y$. Then (rather surprisingly) this is an equivalence
relation. For if $x \sim y \sim z$ and not $x \sim z$, then there must be distinct $a, b \in A$ between~$x$ and~$z$. The
interval between~$x$ and~$z$ cannot be contained in either of the intervals between~$x$ and~$y$ or between~$y$ and~$z$ (since
then $x \sim z$ would be immediate). We assume without loss of generality that $x < y$, and so it follows that $y < z$. One of
$a, b$ lies in $[x,y]$ and the other in $[y, z]$, but now~$a$ and~$b$ are consecutive members of a copy of~$\mathbb Q$, which
is impossible.

The equivalence classes are clearly convex and can intersect~$A$ in at most one point. So this gives two main options, that is,
equivalence classes intersecting~$A$ (in a singleton), which we call `red', and those which are disjoint from~$A$, which we
call `blue'. Furthermore, being convex subsets of~$\mathbb Q$, each $\sim$-class is a non-empty interval of~$\mathbb Q$ of the
form $(a, b)$, $(a, b]$, $[a, b)$, or $[a, b]$ (where $a, b \in \mathbb{R} \cup \{\pm \infty\}$ and in the last case $a = b$
is allowed).

From the definition of~$\sim$ and in particular from convexity of the
equivalence classes we observe the following properties. Putting
$[x]_{\sim}<[y]_{\sim}$ for rationals $x,y$ with $x\not\sim y$ if and
only if $x<y$ gives a well-defined strict linear order on the equivalence
classes. Moreover, by the choice of~$A$, if $x,y\in A$ and $x<y$ then
there are infinitely many points $z\in A$ satisfying $x<z<y$. That is,
between two distinct red classes there are infinitely many red classes.
Similarly, if~$x$ and~$y$ belong to distinct blue equivalence classes,
there must be at least two and thus infinitely many points of~$A$ in
between. Also if~$x$ is in a red class and~$y$ is in a blue one (and we
may assume that actually $x\in A$), then there must be another (and
hence infinitely many) point of~$A$ between them. Consequently, between
any two distinct equivalence classes, we find infinitely many red
classes. Thus, whenever we can show that at least one blue class lies
between any two points of~$A$, then between any two distinct
$\sim$-classes there is a red as well as a blue class.

We denote by~$\Gamma$ the family of all $f \in M$ such that~$\mathbb Q$ may be written as the disjoint union
$\bigcup \{A_q\colon q \in \mathbb{Q}_2\}$ of convex subsets~$A_q$ of~$\mathbb Q$ such that $q < r \Rightarrow A_q < A_r$, each~$A_q$ is
isomorphic to~$\mathbb Q$, and if~$q$ is a red point of $\mathbb{Q}_2$ then~$A_q$ is a red interval of~$\mathbb Q$ with
respect to $\mathop{\mathrm{im}}(f)$ (that is, $|A_q \cap \mathop{\mathrm{im}}(f)| = 1$), and if~$q$ is a blue point of $\mathbb{Q}_2$ then~$A_q$ is blue
(that is, $A_q \cap \mathop{\mathrm{im}}(f) = \emptyset$). The intuition is that the points of the image of~$f$ are spread out as much as
they possibly can be. To handle members of~$M$ whose image may be bounded above or below, we also need to consider~$\Gamma^+$,
which is defined similarly but using $\mathbb{Q}_2 \cup \{\infty\}$ ($\mathbb{Q}_2$ with a right endpoint added), and
similarly~$\Gamma^-$, $\Gamma^\pm$ from $\mathbb{Q}_2 \cup \{-\infty\}$, $\mathbb{Q}_2 \cup \{\pm\infty\}$ (all infinite points coloured blue). The need to consider these variants was pointed out to us by Christian Pech.

A main technical lemma, adapted from~\cite{Bodirsky}, shows how certain pairs of finite partial automorphisms can be extended
to pairs of automorphisms. For this purpose, for any $g \in M$ we let $\sim$ be the equivalence relation defined above with
respect to $\mathop{\mathrm{im}}(g)$, and let~$P$ be the family of all pairs $(a, b)$ of finite partial automorphisms of~$\mathbb Q$
satisfying the following properties:

\begin{enumerate}[(1)]
\item\label{clause:1}
$a$ is colour-preserving (that is, an element $x\in\mathop{\mathrm{dom}}(a)$ belongs to a
red interval if and only if $a(x)$ belongs to a red interval), strongly $\sim$-preserving (meaning that for
$x, y \in \mathop{\mathrm{dom}}(a)$, $x \sim y \Leftrightarrow a(x) \sim a(y)$), and if there is a least or greatest blue interval, then~$a$ preserves it,
\item\label{clause:2}
if $x \in \mathop{\mathrm{dom}}(a)$ lies in a red interval containing a point~$y$ of $\mathop{\mathrm{im}}(g)$, then $y \in \mathop{\mathrm{dom}}(a)$,
\item\label{clause:3}
if $x \in \mathop{\mathrm{im}}(a)$ lies in a red interval containing a point~$y$ of $\mathop{\mathrm{im}}(g)$, then $y \in \mathop{\mathrm{im}}(a)$,
\item\label{clause:4}
$g(\mathop{\mathrm{dom}}(b))\subseteq \mathop{\mathrm{dom}}(a)$,
\item\label{clause:5}
$g(\mathop{\mathrm{im}}(b))\subseteq \mathop{\mathrm{im}}(a)$,
\item\label{clause:6}
if $x \in \mathop{\mathrm{im}}(g) \cap \mathop{\mathrm{dom}}(a)$, then $g^{-1}(x) \in \mathop{\mathrm{dom}}(b)$, and $gbg^{-1}(x) = a(x)$,
\item\label{clause:7}
if $x \in \mathop{\mathrm{im}}(g) \cap \mathop{\mathrm{im}}(a)$, then $g^{-1}(x) \in \mathop{\mathrm{im}}(b)$, and $gb^{-1}g^{-1}(x) = a^{-1}(x)$.
\end{enumerate}

\begin{lemma} \label{2.1} Let $g \in \Gamma \cup \Gamma^+ \cup \Gamma^- \cup \Gamma^\pm$. Then any $(a,b) \in P$ can be
extended to a pair of automorphisms $(\alpha, \beta)$ of\/ $(\mathbb{Q}, < )$ such that $\alpha g = g \beta$.   \end{lemma}
\begin{proof} We first treat the case $g \in \Gamma$. We define a finite partial automorphism~$\overline a$ of $\mathbb{Q}_2$
thus. Let $\mathbb{Q} = \bigcup\{A_q\colon q \in \mathbb{Q}_2\}$ as in the definition of $g \in \Gamma$, and let
${\overline a}(q) = r$ if there is $x \in A_q \cap \mathop{\mathrm{dom}}(a)$ such
that $a(x) \in A_r$. Clause~\eqref{clause:1} guarantees that
$\overline a$ is well-defined, and clauses~\eqref{clause:1}, \eqref{clause:2}, \eqref{clause:3} ensure that it is colour-preserving. Extend~$\overline a$ to an
automorphism~$\overline \alpha$ of $\mathbb{Q}_2$, and let $\alpha \in \mathop{\mathrm{Aut}}(\mathbb{Q}, <)$ be an extension of~$a$
satisfying $\alpha(\mathop{\mathrm{im}}(g))=\mathop{\mathrm{im}}(g)$
(it is possible to achieve this because
$a(\mathop{\mathrm{im}}(g)\cap\mathop{\mathrm{dom}}(a))\subseteq \mathop{\mathrm{im}}(g)$
by~\eqref{clause:6} and
$a^{-1}(\mathop{\mathrm{im}}(g)\cap\mathop{\mathrm{im}}(a))\subseteq \mathop{\mathrm{im}}(g)$
by~\eqref{clause:7})
such that for each $q \in \mathbb{Q}_2$, $\alpha(A_q) = A_{{\overline \alpha}(q)}$. This is possible since
for any~$q$ such that $A_q \cap \mathop{\mathrm{dom}}(a) \neq \emptyset$, if $a(A_q \cap \mathop{\mathrm{dom}}(a)) \subseteq A_r$, then
${\overline a}(q) = r$. Let $\beta = g^{-1}\alpha g$, which is also an
automorphism, since $\alpha$ preserves $\mathop{\mathrm{im}}(g)$.

A similar argument applies to~$\Gamma^+$, $\Gamma^-$, $\Gamma^\pm$, using $\mathbb{Q}_2 \cup \{\infty\}$,
$\mathbb{Q}_2 \cup \{-\infty\}$, $\mathbb{Q}_2 \cup \{\pm \infty\}$ respectively in the argument in place of~$\mathbb{Q}_2$,
noting that the final condition in clause~\eqref{clause:1} ensures that the greatest or least blue interval, if it exists,
is preserved by~$\overline \alpha$.
\end{proof}

\begin{lemma}\label{2.2}
Any injective monoid homomorphism~$\xi\colon M\to E$ which fixes~$G$ pointwise also fixes every member of\/
$\Gamma \cup \Gamma^+ \cup \Gamma^- \cup \Gamma^\pm$.
\end{lemma}
\begin{proof}
Recall the definition of
$S(g) =\{(\alpha,\beta)\in G^2\colon \alpha g = g\beta \}$ for $g\in E$.
Now let $g \in \Gamma$, and consider elements~$u$ and~$s$
of~$\mathbb{Q}$ with $s\neq g(u)$. We construct $(\alpha,\beta)$ in~$S(g)$ such that $\alpha(s) \neq s$ and $\beta(u) = u$. We consider
two cases:

 \begin{enumerate}
  \item If $s\in \mathop{\mathrm{im}}(g)$, then~$s$  and~$g(u)$ lie in different red intervals. Without loss of generality we suppose that $g(u) < s$. Since
  $\mathop{\mathrm{im}}(g) \cong \mathbb{Q}$, there is $t \in \mathop{\mathrm{im}}(g)$ greater than~$s$. Since~$g$ is order-reflecting (that is, its inverse preserves the order), $u < g^{-1}(s) < g^{-1}(t)$. Hence $a = \{(g(u), g(u)), (s, t)\}$ and $b = \{(u, u), ((g^{-1}(s), g^{-1}(t))\}$ are finite partial automorphisms.
We can verify that $(a, b) \in P$ (as defined before Lemma~\ref{2.1}).
  \item If $s \notin \mathop{\mathrm{im}}(g)$, then since $g(u) \neq s$, without loss of generality we suppose that $g(u) < s$. We consider two cases:
  \begin{enumerate}[(i)]
  \item\label{item:blue-interval}
  If~$s$ lies in a blue interval~$A_q$, we choose $t \neq s$ in the same interval.
  Since~$A_q$ is convex,
  $a = \{(s, t), (g(u), g(u))\}$ and $b = \{(u, u)\}$ are finite partial automorphisms. Again
 \mbox{$(a,b) \in P$}.
 \item If~$s$ lies in a red interval~$A_q$ containing $r \in \mathop{\mathrm{im}}(g)$, we choose $t \in A_q\setminus \{g(u),r,s\}$ on the same side of~$r$ (which also allows for the possibility that $r = g(u)$). Then $a = \{(g(u), g(u)),(r,r), (s, t)\}$ and
$b = \{(u, u), (g^{-1}(r), g^{-1}(r))\}$ are finite partial automorphisms, and once more we can verify that $(a,b)\in P$.
 \end{enumerate}
  \end{enumerate}

In each case we can extend $(a,b)$ to $(\alpha,\beta)$ such that $\alpha g = g \beta$ by appealing to Lemma~\ref{2.1}, thus
$(\alpha,\beta)$ lies in $S(g)$, and satisfies $\beta(u) = u$, $\alpha(s) = t \neq s$.

This means that for any~$u$ in~$\mathbb{Q}$ the element~$g(u)$ can be
recovered from~$S(g)$, namely as the unique value~$s$ in~$\mathbb{Q}$
satisfying either side of the equivalence

\begin{equation}\label{Eq:Recover}
g(u)=s\iff \forall (\alpha,\beta)\in S(g)~\left(\beta(u)=u\rightarrow \alpha(s)=s\right)
\end{equation}
For if $g(u)=s$ and $\left(\alpha,\beta\right) \in S(g)$ verifies $\beta(u)=u$, then
$\alpha(s)=\alpha(g(u))=g(\beta(u))=g(u)=s$. This implication is even
true for any~$g\in E$, not just for $g\in\Gamma$.
Conversely, if $g\in\Gamma$ and $g(u)\neq s$, then by the above we can construct $\left(\alpha,\beta\right)\in S(g)$ such that $\beta(u)=u$ and
$\alpha(s)\neq s$.

Note that since~$\xi$ is an injective homomorphism fixing~$G$ pointwise,
\begin{align*}
S(\xi(g)) &= \{(\alpha, \beta) \in G^2\colon \alpha \xi(g) = \xi(g) \beta\} = \{(\alpha, \beta) \in G^2\colon \xi (\alpha g) = \xi(g\beta)\}\\
          &= \{(\alpha, \beta) \in G^2\colon \alpha g = g\beta\} = S(g).
\end{align*}
From this and Condition~\eqref{Eq:Recover} we obtain $\xi(g)=g$: namely,
for $u\in\mathbb{Q}$ put $s:= \xi(g)(u)$, then all of the following
equivalent conditions hold:
\begin{align*}
 \forall (\alpha,\beta)\in S\left(\xi \left(g\right)\right)~
        \left(\beta(u)=u\rightarrow \alpha(s)=s\right)
 & \iff \\
 \forall (\alpha,\beta)\in S(g)~
        \left(\beta(u)=u\rightarrow \alpha(s)=s\right)
 &\stackrel{\text{\eqref{Eq:Recover}}}{\iff} g(u) = s.
\end{align*}

Similar proofs apply in the cases $g \in \Gamma^+$, $\Gamma^-$,
$\Gamma^\pm$. We just note for instance in the case of~$\Gamma^+$ that
if~$s$ lies in the greatest blue interval, then so does~$t$
(Case~\eqref{item:blue-interval}).
\end{proof}

Now we consider how the members of~$\Gamma$ and~$M$ interact. If $g \in
\Gamma$ and $f \in M$ where $\mathop{\mathrm{im}}(f)$ is `coterminal' (that is, for
every $x\in\mathbb{Q}$ there are $u,v\in\mathop{\mathrm{im}}(f)$ with $u\leq x\leq v$),
then any $\sim_{gf}$-class is a union of a convex family of $\sim_g$-classes. This is because $\mathop{\mathrm{im}}(gf) \subseteq \mathop{\mathrm{im}}(g)$
and so if $x \le y$, then
$x \sim_g y \Rightarrow |[x, y] \cap \mathop{\mathrm{im}}(g)| \le 1 \Rightarrow |[x, y] \cap \mathop{\mathrm{im}}(gf)| \le 1 \Rightarrow x \sim_{gf} y$.
Since all $\sim_g$-classes are isomorphic to~$\mathbb Q$, so are all the $\sim_{gf}$-classes. The family of red
$\sim_{gf}$-classes is ordered like~$\mathbb Q$, since it corresponds precisely to the image of~$gf$, which is a copy of~$\mathbb Q$. And
the blue $\sim_{gf}$-classes occupy some cuts among the red ones. Two distinct blue $\sim_{gf}$-classes must occupy distinct
cuts, as if they had no red $\sim_{gf}$-class between them, then by definition of~$\sim_{gf}$, they would have to be in the same
$\sim_{gf}$-class. This means that we may write~$\mathbb Q$ as a disjoint union of sets~$A_q$ for~$q$ lying in some subset~$Q$
of~$\mathbb{Q}_2$, where each~$A_q$ is isomorphic to~$\mathbb Q$ and all the red members of~$\mathbb{Q}_2$ lie in~$Q$. This
describes the general set-up. Depending on the particular~$g$ and~$f$, we may find that $gf \in \Gamma$ or not. We first see
that if they both lie in~$\Gamma$, then the product necessarily does too. Modified remarks apply in the cases where $\mathop{\mathrm{im}}(f)$ is
bounded above, or below, or both, in which case we use the appropriate class, $\Gamma^+$ or~$\Gamma^-$ or~$\Gamma^\pm$.

\begin{lemma}\label{2.3}
If~$g_1$ and~$g_2$ lie in\/~$\Gamma$ then so does~$g_2g_1$ (and similarly for\/~$\Gamma^+$, $\Gamma^-$,
$\Gamma^\pm$).
\end{lemma}

\begin{proof} From the above remarks, we just need to see that between any two $g_2g_1$-red intervals there is a $g_2g_1$-blue
one. Let $g_2g_1(x) < g_2g_1(y)$. Since $g_1 \in \Gamma$, there is a~$g_1$-blue interval $(a, b) \subseteq (g_1(x), g_1(y))$, and its
endpoints~$a$ and~$b$ are irrationals which are limits of points of $\mathop{\mathrm{im}}(g_1)$.
Let $a = \sup_{n\in\mathbb{N}} g_1(a_n)$, $b = \inf_{n\in\mathbb{N}} g_1(b_n)$ where
$(a_n)$ is an increasing sequence, and $(b_n)$ is a decreasing sequence. From $(a, b) \cap \mathop{\mathrm{im}}(g_1) = \emptyset$ it
follows that $g_2(a, b) \cap \mathop{\mathrm{im}}(g_2g_1) = \emptyset$. Let $(c, d)$ be the $g_2g_1$-interval containing $g_2(a, b)$. If
$c \leq  g_2g_1(a_n)$ for some $n\in\mathbb{N}$, then $c \leq g_2g_1(a_n) < g_2g_1(a_{n+1}) < g_2g_1(a_{n+2}) < d$ which would give more than one
point of $\mathop{\mathrm{im}}(g_2g_1)$ in $(c, d)$, contrary to its being a $g_2g_1$-interval. Similarly we cannot have
$g_2g_1(b_n) \leq d$ for any~$n$. Therefore, if $c < g_2g_1(z) < d$ for some $z\in\mathbb{Q}$, then $g_2g_1(a_n) < g_2g_1(z) < g_2g_1(b_n)$
for every~$n$. This implies that $g_1(a_n) < g_1(z) < g_1(b_n)$ for all $n\in\mathbb{N}$, hence $a < g_1(z) < b$, contrary to
$(a, b) \cap \mathop{\mathrm{im}}(g_1) = \emptyset$.
Consequently, $(c,d)\cap \mathop{\mathrm{im}}(g_2g_1) = \emptyset$, and $(c,d)$ is a $g_2g_1$-blue interval.
Furthermore, for~$t$ in $(c,d)$, we have $g_2g_1(a_n)<c<t$ for all~$n$, thus
$t\leq g_2g_1(x)$ would imply $g_1(x)>g_1(a_n)$ for every~$n$ and so $g_1(x) >a$.
This contradicts $(a,b)\subseteq (g_1(x),g_1(y))$, hence $g_2g_1(x) < t$.
Analogously, we can prove $t< g_2g_1(y)$, and therefore, $(c,d)\subseteq (g_2g_1(x),g_2g_1(y))$.

From this and the basic properties of~$\sim_{g_2g_1}$ observed earlier,
it easily follows that the family of $\sim_{g_2g_1}$-intervals is ordered
like~$\mathbb{Q}_2$.
\end{proof}

\begin{lemma}\label{2.4}
  For any $f \in M$ whose image is coterminal in~$\mathbb Q$, there is $g \in \Gamma$ such that
  $gf \in \Gamma$ (with similar statements for the other classes~$\Gamma^+$, $\Gamma^-$, $\Gamma^\pm$).
\end{lemma}

\begin{proof} It is no doubt possible to prove this directly, but it seems a little easier to go by way of the previous lemma.
We start by taking any $g_1 \in \Gamma$, and then we see that we can
describe~$g_1f$ fairly well. Then we take another
$g_2 \in \Gamma$, which will be chosen so that $g_2g_1f \in \Gamma$. Appealing to Lemma~\ref{2.3}, we may let $g = g_2g_1$ to
conclude the proof.

By the discussion above, there is a subset~$Q$ of~$\mathbb{Q}_2$ containing all the red points, such that
$\mathbb{Q} = \bigcup_{q \in Q}A_q$ where the~$A_q$ are copies of~$\mathbb Q$ such that $q < r$ in~$Q$ implies that
$A_q < A_r$ and if $q \in Q$ is red, then~$A_q$ is a $g_1f$-red interval, and if it is blue, then~$A_q$ is a $g_1f$-blue
interval. Let us also write $\mathbb{Q} = \bigcup_{q \in \mathbb{Q}_2}B_q$ where $B_q \cong \mathbb{Q}$ and
$q < r \Rightarrow B_q < B_r$, and we choose $g_2 \in \Gamma$ mapping~$A_q$ to~$B_q$ for each $q \in Q$. More precisely, for
this we let $B_q = \bigcup_{r \in \mathbb{Q}_2}B_{q,r}$ where $B_{q,r} \cong \mathbb{Q}$ and $r < s \Rightarrow B_{q,r} < B_{q,s}$ and ensure that if~$r$ is red, $|\mathop{\mathrm{im}}(g_2) \cap B_{q,r}| = 1$, and if~$r$ is blue,
$\mathop{\mathrm{im}}(g_2) \cap B_{q,r} = \emptyset$. From this we can see that $g_2 \in \Gamma$. Furthermore, each~$B_q$ for red~$q$ is a $g_2g_1f$-red interval, and for blue~$q$ is a $g_2g_1f$-blue interval. Hence also $g_2g_1f \in \Gamma$.
\end{proof}

\begin{corollary}\label{2.5}
  Any injective monoid homomorphism~$\xi\colon M \to E$ which fixes~$G$ pointwise also fixes every member of~$M$.
\end{corollary}
\begin{proof}
Let $f \in M$. By Lemma~\ref{2.4}, if $\mathop{\mathrm{im}}(f)$ is coterminal, then there is $g \in \Gamma$ such that
$gf \in \Gamma$. By Lemma~\ref{2.2}, $\xi$ fixes~$g$ and~$gf$. Therefore
$g\xi(f) = \xi(g)\xi(f) = \xi(gf) = gf$. Since~$g$ is in~$\Gamma$ and
thus in~$M$, it is left cancellable (see Lemma~\ref{3.2} below), and hence $\xi(f) = f$. If $\mathop{\mathrm{im}}(f)$ is bounded above but not
below, we argue similarly using~$\Gamma^+$ in place of~$\Gamma$, and~$\Gamma^-$, $\Gamma^\pm$ correspond in a similar way to
the cases $\mathop{\mathrm{im}}(f)$ bounded below and not above, and bounded above and below, respectively.
\end{proof}

It clearly follows from this corollary that every injective endomorphism
of~$M$ fixing~$G$ pointwise is the identity on~$M$. This implies the
following theorem.

\begin{theorem}\label{2.6}
$M = \mathop{\mathrm{Emb}}(\mathbb{Q}, \le)$ has automatic homeomorphicity, meaning that any isomorphism
between~$M$ and a closed submonoid of the full transformation monoid on a countable set is a homeomorphism.
\end{theorem}
\begin{proof}
This follows from Corollary~\ref{2.5} and~\cite[Lemma~12, p.~13]{Bodirsky},
since~$G$ is dense in the closed monoid $M=\mathop{\mathrm{End}}(\mathbb{Q},<)$; for
by~\cite[Proposition~7, p.~8]{Bodirsky} we know that~$G$ has automatic
homeomorphicity, since it has the small index property~\cite{Truss} and
hence automatic continuity~\cite[3.6, p.~8]{Bodirsky}.
\end{proof}

\section{Preliminary results for the endomorphism monoid of $(\mathbb{Q}, \le)$}
\label{sect:3}

Let us now consider $(\mathbb{Q}, \le)$, and the associated four `natural' monoids, namely its endomorphisms
$\mathop{\mathrm{End}}(\mathbb{Q}, \le)$, embeddings $\mathop{\mathrm{Emb}}(\mathbb{Q},\le)$
(being the same as the injective endomorphisms due to the order being
linear), surjective endomorphisms $\mathop{\mathrm{Surj}}(\mathbb{Q}, \le)$, and
automorphisms \mbox{$\mathop{\mathrm{Aut}}(\mathbb{Q}, \le)$}. The embeddings and automorphisms are the same as for $(\mathbb{Q}, <)$, so we
continue to abbreviate these as~$M$ and~$G$ respectively. The others we write as~$E$ (for endomorphisms) and~$S$
(for `surjective') respectively. Since we want to see what we can `recover' from~$G$ as before, we first look at
which subsets of~$E$ are definable. We starting by showing that the
surjective endomorphisms coincide with the
endomorphisms having a right inverse.

\begin{lemma}\label{3.1}
Each map which is right inverse to some $f\in S$ belongs to~$M$.
In particular a member of~$E$ belongs to~$S$ if and only if it has a right inverse endomorphism.
Furthermore, the sets of right inverse endomorphisms of distinct members
of~$S$ are unequal.
\end{lemma}

\begin{proof}
Let $f \in S$ and suppose $g\colon\mathbb{Q}\to\mathbb{Q}$
satisfies $fg=\mathrm{id}_{\mathbb{Q}}$. We show that
$g\in\mathop{\mathrm{End}}(\mathbb{Q},<) = M$. For this consider
$x,y\in \mathbb{Q}$ such that $x<y$.
Then $f(g(x))=x\not\geq y = f(g(y))$, which implies that
$g(x)\not\geq g(y)$ since~$f$ is order-preserving. As the order is
linear, it follows that $g(x)<g(y)$.

As every $f\in S$ is surjective, it has a right inverse map, which
belongs to $M\subseteq E$ by the above. Moreover, if $f\in E$ has a
right inverse~$g$, then $fg=\mathrm{id}_{\mathbb{Q}}$ implies that~$f$ is surjective.

To prove the final remark, we observe how to `recover' (i.e.\ define) $f \in S$ from its family of right inverses. In fact
$f(x) = y \Leftrightarrow (\exists g\in E)(fg=\mathrm{id}_{\mathbb{Q}} \land g(y) = x)$.
For from the above, it is clear that if $f(x) = y$, there is a right
inverse map (and hence endomorphism) taking~$y$ to~$x$, which gives
`$\Rightarrow$'.
Conversely, if $g(y) = x$ for some right inverse~$g$ of~$f$, then $f(x) = fg(y) = y$.
\end{proof}

Let us write~$C$ for the family of constant maps, namely
$\{c_a\colon a \in \mathbb{Q}\}$, where $c_a(x) = a$ for all~$x$.
Thus $C \subseteq E$ (but of course $C \cap M = \emptyset = C\cap S$).
All the mentioned sets are indeed definable in~$E$.

\begin{lemma} \label{3.2}
Each of~$C$, $M$, $S$, $G$ is a definable subset of~$E$:
$C$ contains precisely all left absorbing (left zero) elements in~$E$,
$M$ are the monomorphisms, $S$ coincides with the epimorphisms, and~$G$ consists of the isomorphisms
inside~$E$.
\end{lemma}

\begin{proof} We have to show that $C = \{g \in E\colon (\forall f \in E)gf = g\}$. To see that this correctly defines~$C$,
first let
$a \in \mathbb{Q}$, and note that for any $f \in E$, $c_af = c_a$ since $c_af(x) = a = c_a(x)$ for all $x\in\mathbb{Q}$. Conversely, suppose that
$gf = g$ for all $f\in E$, and pick any $a \in \mathbb{Q}$. Then $gc_a = g$ and hence for any $x \in \mathbb{Q}$,
$g(x) = gc_a(x) = g(a)$, so~$g$ is constant.

We would like to characterize~$M$ as the members of~$E$ with left inverses, but this is incorrect, as one sees for instance by
considering the function $f(x) = x$ if $x < \pi$, $x + 1$ if $x > \pi$. If this had a left inverse~$g$ say, then for all~$a$
and~$b$ such that $a < \pi < b$, $f(a) < 4 < f(b)$, and so $a < g(4) < b$, which forces $g(4)$ to be~$\pi$ which is not
rational. Instead we use a related condition, of left cancellability (i.e.,
of being a `monomorphism'). So we shall show that a member~$f$ of~$E$ lies in~$M$ if
and only if for any~$g$ and~$h$ in~$E$, $fg = fh \Rightarrow g = h$. If $f \in M$ then this property holds, since for any~$g$
and~$h$ such that $fg(x) = fh(x)$ holds for any~$x$ in~$\mathbb{Q}$, we
have $g(x) = h(x)$ due to~$f$ being injective. Conversely, suppose that~$f$ is a monomorphism in~$E$. Whenever $x,y\in \mathbb{Q}$ are such that
$f(x)=f(y)$, then $fc_x = fc_y$ holds for the constant endomorphisms
$c_x,c_y\in E$. As~$f$ is left cancellable, this implies $c_x = c_y$, and
so $x=y$. Hence, $f$ is an injective endomorphism, thus it belongs to~$M$.

Clearly, by Lemma~\ref{3.1} the set~$S$ is definable as the collection of endomorphisms of $(\mathbb{Q},\leq)$ having a
right inverse endomorphism. However, we can also show that $f \in S$ if and only if it is right cancellable (so is an `epimorphism').
Certainly, if $f\in S$, then it is right cancellable because there is a
right inverse for~$f$.
Conversely, suppose that~$f$ is not surjective, and let~$y$ not lie in
its image. Let $g(x) = h(x) = x$ if $x < y$, $g(x) = h(x) = x + 1$ if $x > y$, and $g(y) = y$, $h(y) = y+1$. Then
$g, h \in E$, and they agree on $\mathbb{Q}\setminus\{y\}\supseteq \mathop{\mathrm{im}}(f)$, and hence $gf = hf$. However,
$g \neq h$, so~$f$ is not right
cancellable.

Finally, $G = M \cap S$, so it too is definable as the set of isomorphisms (i.e.\ the morphisms having two-sided inverses).
\end{proof}

\begin{lemma} \label{3.3} For any $h \in \mathop{\mathrm{End}}(\mathbb{Q}, \le)$ there are $f \in M$ and $g \in S$ such that $h = gf$.
\end{lemma}
\begin{proof} If $q \in\mathbb{Q}$ then $h^{-1}(\{q\})$ is a convex subset of~$\mathbb Q$, since $h(x_1) = h(x_2) = q$ and
$x_1 \le y \le x_2$ imply that $h(y) = q$. Let $X = \bigcup_{q \in \mathbb{Q}}X_q$ where~$X_q$ equals
$\{q\} \times h^{-1}(\{q\})$ if~$q$ lies in the image of~$h$, and is
$\{q\}$ otherwise. We order~$X$ lexicographically, i.e., each~$X_q$ is
ordered as subset of~$\mathbb Q$, and we put $x_1<x_2$ for
$x_1\in X_{q_1}$ and $x_2\in X_{q_2}$ if and only if $q_1 < q_2$.
Then~$X$ is countable densely linearly ordered without endpoints,
so there is an isomorphism $\theta\colon \mathbb{Q} \to X$. Let $\varphi\colon X \to \mathbb{Q}$ be given by $\varphi((q, y)) = q$ if
$h(y) = q$, and $\varphi(q) = q$ if~$q$ does not lie in the image of~$h$. Finally, let $f(x) = \theta^{-1}((h(x), x))$ and
$g = \varphi\theta$.

We verify the desired properties. To see that $f \in M$, let $x < y$. Then $h(x) \le h(y)$ so it follows that
$(h(x), x) < (h(y), y)$, so $\theta^{-1}(h(x), x) < \theta^{-1}(h(y), y)$. Also, since~$\theta$ and $\varphi$ are
order-preserving and surjective, so is~$g$. Finally, to see that $h = gf$, take any $x \in \mathbb{Q}$. Then
$q = h(x) \in \mathop{\mathrm{im}}(h)$, so $(h(x), x) \in X_q$ and $f(x) = \theta^{-1}((h(x), x))$, so
$gf(x) = \varphi\theta \theta^{-1}((h(x), x)) = \varphi((h(x), x)) = h(x)$.     \end{proof}

\begin{corollary} \label{3.4} Any monoid automorphism~$\xi$ of~$E$ which fixes~$G$ pointwise is the identity.  \end{corollary}

\begin{proof} The key point here is that since by Lemma~\ref{3.2} $M$ is definable in~$E$ as the family of left cancellable
elements, $\xi$ must map~$M$ to itself, so we can appeal to Corollary~\ref{2.5} to deduce that it also fixes~$M$ pointwise.
The first part of Lemma~\ref{3.1} implies that~$\xi$ fixes~$S$ setwise, so
it follows from the contrapositive of the second part that it
fixes~$S$ pointwise. Now it is immediate from~Lemma~\ref{3.3} that~$\xi$ fixes
every member of~$E$.
\end{proof}

Note that we would really like this to hold for injective endomorphisms, and not just for automorphisms. This may be true, but
our proof does not show it at present; that is because for a possibly not
surjective~$\xi$, it is not clear that the defining
property of~$M$ inside~$E$ (namely left cancellability) carries over to its
image under~$\xi$. A more detailed analysis of the proof of Lemma~\ref{3.2}
however shows that the property does hold for injective endomorphisms~$\xi$
whose image contains at least one constant operation.

We conclude this section by showing the definability of some other concepts, related to what we have already done.

\begin{lemma} \label{3.5} The relation $f, g \in M \land \mathop{\mathrm{im}}(f)
\subseteq \mathop{\mathrm{im}}(g)$ is definable in the monoid~$E$. \end{lemma}
\begin{proof} We already know that membership in~$M$ is definable. We can then define the given relation by
$(\exists h \in M)f = gh$. Clearly if this formula is true, then the image of~$f$ is contained in the image of~$g$. Conversely,
if $\mathop{\mathrm{im}}(f) \subseteq \mathop{\mathrm{im}}(g)$, we can define~$h$ by $h(q) = r \Leftrightarrow f(q) = g(r)$. This defines~$h$ since
$\mathop{\mathrm{im}}(f) \subseteq \mathop{\mathrm{im}}(g)$, and it is well-defined because~$g$
is 1--1. Finally, $h$ preserves the (strict) order since~$f$ does and~$g$ reflects it.
\end{proof}

This result may be used to give a `representation' of~$\mathbb Q$ inside~$M$,
namely we can characterize those members~$f$ of~$M$
whose image omits precisely one point of~$\mathbb Q$ by the formula
$f \in M\setminus G \land (\forall g\in M\setminus G (\mathop{\mathrm{im}}(f) \subseteq
\mathop{\mathrm{im}}(g) \to \mathop{\mathrm{im}}(g) \subseteq \mathop{\mathrm{im}}(f)))$, representing
that~$f$ has a maximal image among non-automorphisms. And of course we can also
characterize when two such maps `encode' the same point by saying that they have the same image.

We remark that in~$E$, by contrast, we already have the constant maps~$c_q$ available, so we have an immediate and direct way
of representing the points of~$\mathbb Q$ inside the monoid.

Finally in this section, we show how finite subsets of~$\mathbb Q$ can be represented in~$E$.

\begin{lemma}\label{3.6}
For any $f \in E$,
$\mathop{\mathrm{im}}(f) = \{q \in \mathbb{Q}\colon (\exists h\in E)fh = c_q\}$.
Hence $|\mathop{\mathrm{im}}(f)| = n \Leftrightarrow $ there are exactly~$n$
constants~$k$ such that $(\exists h\in E)fh = k$.
\end{lemma}

\begin{proof}
If $q=f(r)$ for some $r\in\mathbb{Q}$, then $fc_r = c_q$, so we may
choose $h\in E$ as~$c_r$. Conversely, if $fh=c_q$ for some
$h\in E$, then $\{q\} = \mathop{\mathrm{im}}(c_q)=\mathop{\mathrm{im}}(fh)\subseteq
\mathop{\mathrm{im}}(f)$.
\end{proof}

We remark that the situation for these maps is radically different in the cases $n = 1$ and $n > 1$. For $n = 1$ there are
exactly $\aleph_0$ maps having image of that size, namely the constant maps~$c_q$. But if $n > 1$, for each~$B$ of size~$n$
there are $2^{\aleph_0}$ maps having image~$B$. For if $B = \{b_0, b_1, \dotsc, b_{n-1}\}$ then $f^{-1}(\{b_i\})$ are pairwise
disjoint intervals with endpoints $a_i, a_{i+1}$ say, $-\infty = a_0 < a_1 < \ldots < a_n = \infty$ (open or closed or
semi-open) and~$a_i$ may take any real value. All the same, these maps are quite easy to visualize, and will play an important
part in what follows.

\section{Automatic homeomorphicity of $\mathop{\mathrm{End}}(\mathbb{Q}, \le)$}
\label{sect:4}

In this section we give a discussion of the automatic homeomorphicity question for~$E$. Here, since~$G$ is not dense in~$E$,
we are obliged to use a more direct method, which may be of some
independent interest (and will also be used in section~\ref{sect:5}). In
the hypothesis of automatic homeomorphicity we are asked to consider an isomorphism~$\theta$ of~$E$ with a closed
submonoid~$E'$ of the full transformation monoid $\mathop{\mathrm{Tr}}(\Omega)$ on some countable set~$\Omega$, and show that it is a homeomorphism. This
$\theta$ may be viewed as a (faithful) monoid action of~$E$ on~$\Omega$ (which we write as a left action). Our strategy is
to try to demonstrate directly that~$\theta$ is a homeomorphism, by describing explicitly what it can be. To that end, let us
study the $G$-orbits of~$\theta$. If $X\subseteq \Omega$ is one such orbit, then for some $x \in X$, $X = \{\theta(g)(x)\colon g \in G\}$. By the
orbit-stabilizer theorem, the orbit is in natural 1--1 correspondence with the left cosets of the stabilizer
$G_x = \{g \in G\colon \theta(g)(x) = x\}$. Since $X \subseteq \Omega$, it is countable, and so $|G\mathbin{:} G_x|$ is countable. By the
small index property for~$G$~\cite{Truss}, $G_x = G_B$ for some finite $B \subseteq \mathbb{Q}$, and furthermore, this gives
rise to an identification of~$X$ with the set $[\mathbb{Q}]^n$ of the $n$-element subsets of~$\mathbb Q$ respecting the
action as follows: Let $a_{g(B)} = \theta(g)(x)$. Then
$a_{g_1(B)} = a_{g_2(B)} \Leftrightarrow \theta(g_1)(x) = \theta(g_2)(x)
\Leftrightarrow g_2^{-1}g_1 \in G_x \Leftrightarrow g_2^{-1}g_1 \in G_B
\Leftrightarrow g_1(B) = g_2(B)$. Since $[\mathbb{Q}]^n$, the set of $n$-element subsets of~$\mathbb Q$,
forms an orbit under the action of~$G$, this means that we may write~$X$ as
$\{a_{g(B)}\colon g \in G\} = \{a_C\colon C \in [\mathbb{Q}]^n\}$, and the action of~$\theta$ is given by $\theta(g)(a_C) = a_{g(C)}$.
Under these circumstances we say that this $G$-orbit has \emph{rank}~$n$.

The conclusion of the discussion in the previous paragraph is that~$\Omega$ may be written as the union of $G$-orbits, each
having finite rank, and~$\theta$ provides a natural action of~$G$ on each $G$-orbit. Let us write
$\Omega = \bigcup_{i \in I}\Omega_i$, where~$\Omega_i$ are the $G$-orbits, and let~$\Omega_i$ have rank~$n_i$, so that we may
write $\Omega_i = \{a^i_B\colon B \in [\mathbb{Q}]^{n_i}\}$. The action is therefore given by $\theta(g)(a^i_B) = a^i_{g(B)}$ for
each $i \in I$ and $B \in [\mathbb{Q}]^{n_i}$. What we now want to do is to show how this action extends to an action of~$E$,
first treating members of~$M$. To do this, we need to know that the
restriction $\theta\restriction_M\colon M\to M'$, where $M'=\theta(M)$,
is continuous. We could infer this from Theorem~\ref{2.6} once we knew
that~$M'$ is a closed submonoid of~$E'$. However, it turns out we first
need to prove continuity of the restriction before we can verify this
assumption, so using Theorem~\ref{2.6} does not seem to be the right way
to do it.

\begin{lemma}\label{4.0}
For an isomorphism $\theta\colon E\to E'$ to a closed submonoid
$E'\subseteq \mathop{\mathrm{Tr}}(\Omega)$ on a countable set~$\Omega$,
the monoid $M' = \theta(M)$ is closed in~$\mathop{\mathrm{Tr}}(\Omega)$
and the restriction $\theta\restriction_M\colon M\to M'$ is a
homeomorphism.
\end{lemma}
\begin{proof}
This is an almost verbatim copy of the proof of Lemma~12
in~\cite{Bodirsky}, but with the ending modified as we are in a slightly
different situation.

Let us denote by~$G'$ the monoid reduct of the group of invertible
elements of~$E'$, and let~$\overline{G'}$ be the closure of~$G'$ in~$E'$;
this is again a transformation monoid, and~$G'$ is dense in it. We also know that~$G$
comprises the set of invertible elements of~$E$, and it is dense in the
closed monoid~$M$. It is easy to see that $\theta(G)\subseteq G'$
as~$\theta$ is a monoid homomorphism. Moreover, since~$\theta$ is an
isomorphism, $\theta^{-1}(G')\subseteq G$ follows by a symmetric
argument, and hence $\theta(G)=G'$ so that the restriction
$\theta\restriction_G\colon G \to G'$ is a well-defined bijective monoid
homomorphism. As the monoids~$G$ and~$G'$ are group reducts,
$\theta\restriction_G$ actually is a group isomorphism, too. Moreover,
density of~$G'$ in the closed monoid~$\overline{G'}$ implies that
$G'= \overline{G'}\cap \mathop{\mathrm{Sym}}(\Omega)$, and similarly~$G$
is a closed subgroup of the full symmetric group on~$\mathbb{Q}$.
As the automorphism group~$G$ has automatic homeomorphicity,
$\theta\restriction_G$ is a homeomorphism. Now applying Proposition~11
of~\cite{Bodirsky}, there is an extension
$\overline{\theta\restriction_G}\colon M\to \overline{G'}$
of~$\theta\restriction_G$, which is a monoid isomorphism and a
homeomorphism. As~$E'$ is closed, $\overline{G'}\subseteq E'$, and we let $\iota\colon \overline{G'}\to E'$ be the
inclusion map, which is a monoid embedding. Then $\xi:= \theta^{-1}\iota\overline{\theta\restriction_G}$
is an injective monoid homomorphism from~$M$ into~$E$, which fixes every
member of~$G$. By Corollary~\ref{2.5},
$\xi(f) = f$ for every $f\in M$, i.e.\
$\theta(f) = \theta(\xi(f)) = \iota(\overline{\theta\restriction_G}(f))
                            = \overline{\theta\restriction_G}(f)$.
This proves that
$M'= \theta(M) = \overline{\theta\restriction_G}(M) = \overline{G'}$,
and hence~$M'$ is closed in~$\mathop{\mathrm{Tr}}(\Omega)$. Moreover,
$\theta\restriction_M\colon M\to M'$ coincides
with~$\overline{\theta\restriction_G}$ and consequently is a
homeomorphism.
\end{proof}

\begin{lemma}\label{4.1}
For each $f \in M$, $i \in I$, and
$a^i_B \in \Omega$, $\theta(f)(a^i_B) = a^i_{f(B)}$.
\end{lemma}

\begin{proof} As~$G$ is dense in~$M$, we may find a sequence $(g_n)$ in~$G$ such that $g_n \to f$. Now the topologies on~$M$
and~$M'$ are generated by sub-basic open sets of the form $\mathcal{B}_{qr} = \{h \in M\colon h(q) = r\}$ for
$q, r \in \mathbb{Q}$ and $\mathcal{C}_{ijBC} = \{h \in \theta(M)\colon h(a^i_B) = a^j_C\}$ for
$i, j \in I$, $B\in [\mathbb{Q}]^{n_i}$ and $C\in[\mathbb{Q}]^{n_j}$. Let $B = \{q_1, \ldots, q_m\}$ and $r_k = f(q_k)$. Since $g_n \to f$
and $f \in \mathcal{B}_{q_kr_k}$, there is~$N_k$ such that $(\forall n \ge N_k)g_n \in \mathcal{B}_{q_kr_k}$, so for all
$n \ge \max_{1 \le k \le m}N_k$, $g_n(B) = f(B)$. By Lemma~\ref{4.0}, the
restriction of~$\theta$ to~$M$ is continuous.
Hence $\theta(g_n) \to \theta(f)$. Let $\theta(f)(a^i_B) = a^j_C$. Thus
$\theta(f) \in \mathcal{C}_{ijBC}$. From $\theta(g_n) \to \theta(f)$ it follows that
$(\exists N)(\forall n \ge N)\theta(g_n) \in \mathcal{C}_{ijBC}$. Hence for this~$N$,
$(\forall n \ge N)\theta(g_n)(a^i_B) = a^j_C$. But we know that $\theta(g_n)(a^i_B) = a^i_{g_n(B)}$ as $g_n \in G$. Hence for
such~$n$, $j = i$ and $g_n(B) = C$. Taking $n \ge N, \max_{1 \le k \le m}N_k$, it follows that $j = i$ and $C = g_n(B) = f(B)$. Thus
$\theta(f)(a^i_B) = a^i_{f(B)}$ as required. \end{proof}

We can extend the statement of Lemma~\ref{4.1} to certain members of~$E$, provided that they act `like' members of~$M$ on the relevant set.

\begin{lemma}\label{4.2}
If $f \in E$, $i \in I$, and
$a^i_B\in\Omega_i$, where $|f(B)| = n_i = |B|$, then
$\theta(f)(a^i_B) = a^i_{f(B)}$.
\end{lemma}

\begin{proof}
First consider the case where $f \in S$. As in the proof of
Lemma~\ref{3.1} there is a right inverse $h \in M$
for~$f$, and in addition, $h$ may be chosen so that for each $x \in B$,
$hf(x) = x$. Then, applying Lemma~\ref{4.1} to~$h\in M$,
$\theta(f)(a^i_B) = \theta(f)(a^i_{hf(B)}) = \theta(f)\theta(h)(a^i_{f(B)}) =\theta(\mathrm{id}_{\mathbb{Q}})(a^i_{f(B)}) = a^i_{f(B)}$.
Now consider any $h \in E$ such that $|h(B)| = n_i$.
By Lemma~\ref{3.3}, we may write $h = gf$ where $f \in M$ and $g \in S$,
and $|g(f(B))| = |h(B)| = n_i$. Hence
by what we have just shown, $\theta(g)(a^i_{f(B)}) = a^i_{gf(B)}$, so,
by Lemma~\ref{4.1} applied to $f\in M$,
$\theta(h)(a^i_B) = \theta(g)\theta(f)(a^i_B) = \theta(g)(a^i_{f(B)}) =
a^i_{gf(B)} = a^i_{h(B)}$.
\end{proof}

If $f \in E$ `collapses' a set~$B$, then we can certainly not deduce that $\theta(f)(a^i_B) = a^j_C$ for $j = i$,
since~$\Omega_i$ and~$\Omega_j$ will have different ranks. For the proof of openness in the main theorem, we would still need some
information about~$C$, namely that it is contained in~$f(B)$.

\begin{lemma}\label{4.3}
Let $i\in I$ and $B\in[\mathbb{Q}]^{n_i}$.
Then the following statements hold.
\begin{enumerate}[{\upshape(i)}]
\item\label{4.3.1}
  There is an idempotent endomorphism $h\in E$ having~$B$ as image
  such that $\theta(h)(a^i_B) = a^i_B$.
\item\label{4.3.2}
  $\theta(f_1)(a^i_B) = \theta(f_2)(a^i_B)$ whenever $f_1,f_2\in E$
  satisfy $f_1\restriction_B = f_2\restriction_B$.
\item\label{4.3.3}
  If, for $f\in E$, $j\in I$ and $C\subseteq\mathbb{Q}$ are given by
  $\theta(f)(a^i_B)=a^j_C$, then $C\subseteq f(B)$.
\end{enumerate}
\end{lemma}
\begin{proof}
\begin{enumerate}[(i)]
\item
By subdividing~$\mathbb{Q}$ into~$|B|$ pairwise disjoint intervals each
containing a unique member of~$B$, and mapping the whole
of each such interval to the member of~$B$ it contains, we obtain an
endomorphism $h\in E$ fixing all elements of~$B$ and satisfying
$\mathop{\mathrm{im}}(h) = B$, which is clearly idempotent. Since $h(B) = B \in [\mathbb{Q}]^{n_i}$, we can
apply Lemma~\ref{4.2} to get $\theta(h)(a^i_B) = a^i_{h(B)} = a^i_B$.
\item
Consider the idempotent $h\in E$ constructed in~\eqref{4.3.1}.
We see by inspection that $f_1h = f_2h$, wherefore
$\theta(f_1)(a^i_B) = \theta(f_1)\theta(h)(a^i_B) = \theta(f_1h)(a^i_B)
=\theta(f_2h)(a^i_B)= \theta(f_2)\theta(h)(a^i_B) = \theta(f_2)(a^i_B)$.
\item
Now suppose for a contradiction that there is $c \in C \setminus f(B)$. Then
there is $h \in G$ moving~$c$ to $h(c)\notin C$ but fixing
all members of~$f(B)$. Hence $f \restriction_B = f' \restriction_B$,
where $f' = hf$, since~$h$ fixes~$f(B)$ pointwise.
As shown in~\eqref{4.3.2},
$\theta(f')(a^i_B) = \theta(f)(a^i_B) = a^j_C$. However,
$\theta(f')(a^i_B) = \theta(hf)(a^i_B) = \theta(h)\theta(f)(a^i_B) =
\theta(h)(a^j_C) = a^j_{h(C)}$, contrary to $h(C) \neq C$.
We conclude that $C \subseteq f(B)$ as required.
\qedhere
\end{enumerate}
\end{proof}

Using the ideas from above, we can demonstrate automatic homeomorphicity
of $E = \mathop{\mathrm{End}}(\mathbb{Q}, \le)$.

\begin{theorem}\label{4.4}
$E = \mathop{\mathrm{End}}(\mathbb{Q}, \le)$ has automatic homeomorphicity,
meaning that any isomorphism~$\theta$ between~$E$ and a closed
submonoid~$E'\subseteq \mathop{\mathrm{Tr}}(\Omega)$ on a countable set\/~$\Omega$ is a homeomorphism.
\end{theorem}
\begin{proof} The sub-basic open sets in~$E$ and~$E'$ are of the form
$\mathcal{B}_{qr} = \{f \in E\colon f(q) = r\}$ and
$\mathcal{C}_{ijBC} = \{f \in E'\colon f(a^i_B) = a^j_C\}$ for $B \in \Omega_i, C \in \Omega_j$, so to establish continuity we have to
show that each $\theta^{-1}(\mathcal{C}_{ijBC})$ is open in~$E$. Now~$B$ is a finite set, so we may let
$B = \{q_1, q_2, \ldots, q_m\}$, and, for an arbitrary
$f\in\theta^{-1}(\mathcal{C}_{ijBC})$, we let $r_k = f(q_k)$. Thus $f \in
\bigcap_{k=1}^m \mathcal{B}_{q_k r_k}$. We show that
$\bigcap_{k=1}^m \mathcal{B}_{q_k r_k} \subseteq \theta^{-1}(\mathcal{C}_{ijBC})$, and this is what is required, since it shows that
$\theta^{-1}(\mathcal{C}_{ijBC})$ is a union of open sets, hence open in~$E$. For let $f' \in \bigcap_{k=1}^m\mathcal{B}_{q_k r_k}$.
Then $f'(q_k) = r_k$ for each~$k$, which means that~$f$ and~$f'$ agree
on~$B$. By part~\eqref{4.3.2} of Lemma~\ref{4.3}, it follows that
$\theta(f')(a^i_B) = \theta(f)(a^i_B) = a^j_C$. Hence $f' \in
\theta^{-1}(\mathcal{C}_{ijBC})$.

To show that~$\theta$ is also open, it suffices to show that the image of any sub-basic open set is open. So consider
$\theta(\mathcal{B}_{qr})$ for any rationals~$q$ and~$r$. Look at any member of this set, which may be written as $\theta(f)$
where $f \in \mathcal{B}_{qr}$; we shall find $i, j\in I$ and
$B, C\subseteq\mathbb{Q}$ so that
$\theta(f) \in \mathcal{C}_{ijBC} \subseteq \theta(\mathcal{B}_{qr})$. Since $f \in \mathcal{B}_{qr}$, $f(q) = r$. We shall show that
there is some $i \in I$ such that $|\mathop{\mathrm{im}}(f)| \ge n_i > 0$. Then we can find~$B$ and~$C$ of size~$n_i$ such that $f(B) = C$
with $q \in B$. Now we take $j = i$, and observe using
Lemma~\ref{4.2} that, $\theta(f)(a^i_B) = a^i_{f(B)} =  a^i_C$, which tells us
that $\theta(f) \in \mathcal{C}_{iiBC}$. Furthermore, for any $g \in \mathcal{C}_{iiBC}$, since we are in~$E'$, $g = \theta(h)$ for
some $h \in E$, and $\theta(h)(a^i_B) = a^i_C$. Therefore,
Lemma~\ref{4.3}\eqref{4.3.3} yields $C \subseteq h(B)$, and as $|B| = |C|$,
finiteness of~$B$ implies $C = h(B)$. As~$f$ maps~$q$ to~$r$,
and so~$q$ and~$r$ are the corresponding entries of~$B$ and~$C$ when enumerated in increasing order, it follows that~$h$
also maps~$q$ to~$r$. Hence $h \in \mathcal{B}_{qr}$, which shows that $g = \theta(h) \in \theta(\mathcal{B}_{qr})$, as required.

To see that such $i\in I$ exists, suppose otherwise. This means that for every $i \in I$, if $n_i > 0$ then $|\mathop{\mathrm{im}}(f)| < n_i$.
Consider any $i\in I$ and $a^i_B \in \Omega_i$ and let $\theta(f)(a^i_B) =
a^j_C$. Then $C \subseteq f(B)$  by Lemma~\ref{4.3} and so $n_j \le |\mathop{\mathrm{im}}(f)|$. It
follows that $n_j = 0$, and $C = \emptyset$.
Choose $g\in G$ such that $g(f(x))\neq f(x)$ holds for some
$x\in\mathbb{Q}$, e.g.\ $g(y)= y+1$ for $y\in \mathbb{Q}$.
For every $i\in I$ and $B\in [\mathbb{Q}]^{n_i}$,
$\theta(gf)(a^i_B) = \theta(g)\theta(f)(a^i_B) =
\theta(g)(a^j_{\emptyset}) = a^j_{g(\emptyset)} = a^j_{\emptyset} =
\theta(f)(a^i_B)$, showing that $\theta(gf) = \theta(f)$. However,
$gf\neq f$ by the choice of~$g$, contrary to the injectivity of~$\theta$.
\end{proof}

We would like to have more precise information about the action of the image of~$\theta$ on~$\Omega$. We have partial
information about this from Lemmas~\ref{4.1}, \ref{4.2}, \ref{4.3} but this still seems to leave many options open. We now set about
describing the most general situation we are aware of under which there is such an action. The conjecture will then be that
this describes everything that actually can occur.

The most natural way for~$E$ to act is just directly on~$\mathbb Q$, and we can see that the action on any orbit of rank 1
must be like this, since the condition used in Lemma~\ref{4.2} (namely that $|f(B)| = |B|$) is immediately verified. The next most natural action
is on $\bigcup_{1 \le i \le n}[\mathbb{Q}]^i$ for some fixed $n \ge 1$. Here we just let $\theta(f)(B) = f(B)$, and note that
the $G$-orbits are the $[\mathbb{Q}]^i = \{B \subset \mathbb{Q}\colon |B| =
i\}$, and this `cascades' through the $G$-orbits
depending on the behaviour of the map~$f$. Generalizing this, let $n = n_k > n_{k-1} > \ldots > n_0 = 0$. This time we
let $X = \bigcup_{0 \le i \le k}[\mathbb{Q}]^{n_i}$, and define $\theta(f)(B)$ to be the first~$n_i$ elements of~$f(B)$ if~$i$
is greatest such that $n_i \le |f(B)|$, if $i > 0$, and 0 if $i = 0$. It is straightforward to verify that this is an action.

The general action that we have in mind is built up from ones of this kind using a `tree'. The tree in question will be a
countable partially ordered set $(T, \le)$ in which for each $t \in T$, $\{s \in T\colon s \le t\}$ is a finite linearly ordered
set, with a labelling $l\colon T \to \mathbb{N}$ such that $t_1 < t_2 \Rightarrow l(t_1) < l(t_2)$ (strictly speaking, this
is a `forest'). Given such~$T$, which has at least one point labelled by a non-zero number, (or else, infinitely many labelled
0), we can form $\Omega = \bigcup\{[\mathbb{Q}]^{l(t)} \times \{t\}\colon t \in T\}$, and the action is given as above `down each
branch'. That is, $\theta(f)(B_1, t_1) = (B_2, t_2)$ if~$B_2$ is the first $l(t_2)$ elements of $f(B_1)$ if~$t_2$ is the
greatest point below~$t_1$ in~$T$ such that $l(t_2) \le |f(B_1)|$. This is similarly easily verified to be an action. So the
main question remaining here is whether all such actions are of this form.

\section{Automatic homeomorphicity of $\mathop{\mathrm{Pol}}(\mathbb{Q}, \le)$}
\label{sect:5}
In this section we use ideas from earlier in the paper to prove automatic homeomorphicity for the polymorphism clone
$\mathop{\mathrm{Pol}}(\mathbb{Q}, \le)$. For definitions of the relevant notions here we refer the reader to~\cite{Bodirsky}, but mention a few
notations and ideas that are needed.
Denoting by $\mathcal{O}_{A}$ the collection of all finitary
operations $f\colon A^n\to A$ ($n\geq 0$) on a set~$A$, a subset
$C\subseteq \mathcal{O}_{A}$ is called a (`concrete') clone on~$A$
if it is closed under the operations of composition when defined
(that is, the `arities' are correct) and it contains all
`projections'. These are the maps
$\pi_i^{(n)}\colon A^n \to A$ given by $\pi_i^{(n)}(a_1, a_2, \ldots, a_n) = a_i$, where $1 \le i \le n$.
The collection of all polymorphisms of a relational
structure always forms a clone, and clones arising in this way are
precisely the ones that are topologically closed. Of central interest here is the clone
$\mathop{\mathrm{Pol}}(\mathbb{Q},\leq)$ of polymorphisms of $(\mathbb{Q}, \le)$,
which is the family of all $n$-ary
functions on~$\mathbb Q$ for $n \ge 0$ that preserve~$\le$, i.e.\ that
are monotone maps from $(\mathbb{Q},\leq)^n$ to $(\mathbb{Q},\leq)$. Spelling out precisely what this means, $f\colon\mathbb{Q}^n \to \mathbb{Q}$ lies in the clone provided
that if $(a_1, a_2, \ldots, a_n), (b_1, b_2, \ldots, b_n) \in \mathbb{Q}^n$ and $a_i \le b_i$ for all~$i$, then
$f(a_1, a_2, \ldots, a_n) \le f(b_1, b_2, \ldots, b_n)$. There is a corresponding notion of `abstract clone', which we do not
require here. Let us note also that the set~$\mathcal{O}_{A}$ of all
finitary operations on~$A$ forms a clone, even a polymorphism clone
(e.g., $\mathcal{O}_{A} = \mathop{\mathrm{Pol}}(A, A)$).
This is the analogue of $\mathop{\mathrm{Sym}}(A)$ for the automorphism group
and $\mathop{\mathrm{Tr}}(A)$ for the endomorphism monoid.

Relying on results of~\cite{Bodirsky}, when proving automatic
homeomorphicity of the clone $\mathop{\mathrm{Pol}}(\mathbb{Q},\leq)$, it
will suffice to verify that any clone isomorphism between
$\mathop{\mathrm{Pol}}(\mathbb{Q},\leq)$ and a closed clone on some countable set
is continuous. To exhibit the general method we are using here, we first
prove the following result, which is based on adapting the strategy used
to demonstrate the first part of Theorem~\ref{4.4}.

\begin{lemma}\label{5.1}
Let~$A$ and~$B$ be sets, $P$ and~$P'$ be clones on~$A$ and~$B$,
respectively, and $\theta\colon P\to P'$ be a clone homomorphism.
If for every $b\in B$ there is some unary function $h\in P^{(1)}$
with finite image such that $\theta(h)(b) = b$, then~$\theta$ is
continuous.
\end{lemma}
\begin{proof}
Under the given assumptions we have to verify that $\theta^{-1}(C)$ is
open in~$P$ for any sub-basic open set~$C$ of~$P'$. By definition of the
topology of~$P'$ there are $n\in\mathbb{N}$, $(b_1,\dotsc,b_n)=b\in B^n$ and $b'\in B$
such that $C = \{g\in P'^{(n)}\colon g(b)=b'\}$. We want to prove that
every $f\in \theta^{-1}(C)$ is surrounded by a whole open neighbourhood
inside $\theta^{-1}(C)= \{f\in P^{(n)}\colon \theta(f)(b)=b'\}$,
showing that~$f$ is an interior point of~$\theta^{-1}(C)$.

By the assumption of the lemma, we can find maps
$h_1,\dotsc,h_n\in P^{(1)}$
satisfying $\theta(h_i)(b_i) = b_i$ and having finite image
$\mathop{\mathrm{im}}(h_i)\subseteq A$ for every index $i\in\{1,\dotsc,n\}$.
Therefore, the Cartesian product
$A' = \prod_{i=1}^n \mathop{\mathrm{im}}(h_i)\subseteq A^n$ is
finite, too and thus the set
  $P_f = \bigcap_{a\in A'}\{f'\in P^{(n)}\colon f'(a) = f(a) \}$
is a basic open neighbourhood of~$f$ in the topology of~$P$.
Hence, the result is proved once we establish that
$P_f\subseteq \theta^{-1}(C)$.

For this let~$f'$ be any function
in~$P_{f}$, that is, we assume $f'(a) = f(a)$ for every $a\in A'$.
Thus the $n$-ary functions~$f$ and~$f'$ coincide on the finite set
$A'=\prod_{i=1}^{n} \mathop{\mathrm{im}}(h_i)$, which then implies the equation
$f \circ \left(h_1 \circ \pi_1^{(n)}, h_2 \circ \pi_2^{(n)}, \ldots, h_n \circ \pi_n^{(n)}\right) =
f' \circ \left(h_1 \circ \pi_1^{(n)}, h_2 \circ \pi_2^{(n)}, \ldots, h_n \circ \pi_n^{(n)}\right)$.
From here we can conclude that $\theta(f')(b) = \theta(f)(b) = b'$,
i.e., $f'\in \theta^{-1}(C)$, as follows:
\begin{align*}
\theta(f)(b) = \theta(f)(b_1,\dotsc,b_n)
&= \theta(f)(\theta(h_1)(b_1),\dotsc,\theta(h_n)(b_n))\\
&= \theta(f)\left(\theta(h_1)\bigl(\pi_1^{(n)}(b)\bigr),\dotsc,\theta(h_n)\bigl(\pi_n^{(n)}(b)\bigr)\right)\\
&= \theta(f)\left(\theta(h_1)\bigl(\theta\bigl(\pi_1^{(n)}\bigr)(b)\bigr),\dotsc,\theta(h_n)\bigl(\theta\bigl(\pi_n^{(n)}\bigr)(b)\bigr)\right)\\
&= \theta(f)\circ\left(\theta(h_1)\circ\theta\bigl(\pi_1^{(n)}\bigr),\dotsc,\theta(h_n)\circ\theta\bigl(\pi_n^{(n)}\bigr)\right)(b)\\
&= \theta\left(f\circ\left(h_1\circ\pi_1^{(n)},\dotsc,h_n\circ\pi_n^{(n)}\right)\right)(b).
\end{align*}
Similarly, $\theta(f')(b) = \theta\left(f'\circ\left(h_1\circ\pi_1^{(n)},\dotsc,h_n\circ\pi_n^{(n)}\right)\right)(b)$. From
the above equation it follows that $\theta(f')(b) = \theta(f)(b) = b'$, as required.
\end{proof}

Proving automatic homeomorphicity of
$P =\mathop{\mathrm{Pol}}(\mathbb{Q}, \leq )$ now basically boils down
to verifying the assumptions of the preceding result.

\begin{theorem}\label{5.2}
$\mathop{\mathrm{Pol}}(\mathbb{Q}, \le)$ has automatic homeomorphicity, meaning
that any isomorphism~$\theta$ from $P = \mathop{\mathrm{Pol}}(\mathbb{Q}, \le)$
to a closed subclone~$P'$ of\/~$\mathcal{O}_\Omega$, for a countable
set\/~$\Omega$, is a homeomorphism.
\end{theorem}
\begin{proof} Note that, unlike in the case of the \emph{monoid}~$E$, where we would have had to prove both continuity and
openness of the given isomorphism~$\theta$, here we only need to check continuity, since openness follows from Proposition~27
of~\cite{Bodirsky}, and this avoids the need for proving the analogue of Lemma~\ref{4.3} (though this analogue still holds).

To demonstrate that~$\theta$ is continuous, we use the machinery from
section~\ref{sect:4} to provide the assumptions of
Lemma~\ref{5.1}. Note that these properties are determined entirely by the
restriction $\theta\restriction_E\colon P^{(1)}\to P'^{(1)}$, which is a
monoid isomorphism between the unary parts $P^{(1)} = E$ and
$E':= P'^{(1)}$ (these are closed monoids because~$P$ and
$\mathop{\mathrm{Tr}}(\mathbb{Q})$, and~$P'$ and $\mathop{\mathrm{Tr}}(\Omega)$ are closed
sets). Namely, we have to verify that for every $b\in\Omega$ we
can  find an endomorphism $h\in E$ with finite image such that
$\theta(h)(b) = \theta\restriction_E(h)(b)= b$. However, this is
precisely the content of part~\eqref{4.3.1} of Lemma~\ref{4.3} applied
to~$\theta\restriction_E$.
 \end{proof}

\section{Automatic homeomorphicity of clones generated by monoids}
\label{sect:6}

In this final section we show that automatic homeomorphicity results can be lifted from monoids to the polymorphism clones
they generate, under appropriate conditions. Given a submonoid~$E$ of the full transformation monoid $\mathop{\mathrm{Tr}}(\Omega)$ on
a set~$\Omega$, there is a least clone $\langle E \rangle$ on~$\Omega$ containing~$E$; it may be formed by including all
projections, and then closing up under compositions of functions where these are defined; it may be explicitly written as
$\bigcup_{k\in \mathbb{N} \setminus \{0\}}\{f\circ \pi_j^{(k)} \colon j\in \{1,\ldots,k\} \land f \in E\}$. This is of course a
rather small subclone of $\mathop{\mathrm{Pol}}(\Omega)$, so any results obtained about it do not really give us information about the
general situation. Our main result here is that if $E \subseteq \mathop{\mathrm{Tr}}(\Omega)$ is a closed transformation monoid which
has automatic homeomorphicity and its group of invertible members acts transitively on~$\Omega$, then $\langle E \rangle$
also has automatic homeomorphicity.

Because the definition of automatic homeomorphicity, as described in
section~\ref{sect:5}, is given for closed clones we start by
recalling that $\langle E \rangle$ is closed. This result belongs to the folklore of clone theory.

\begin{lemma}\label{6.1} If~$A$ is a set, and $E \subseteq \mathop{\mathrm{Tr}}(A)$ is a closed transformation monoid, then the clone
$\langle E \rangle$ is also closed.
\end{lemma}
\begin{proof} Consider the quaternary relation $\rho = \{(x,y,z,u)\in A^4 \colon x = y \lor z = u\}$. We see that
$\mathop{\mathrm{Pol}}(A,\{\rho,\emptyset\}) = \langle \mathop{\mathrm{Tr}}(A) \rangle$. Since every function in
$\mathop{\mathrm{Pol}}\left(A,\{\rho,\emptyset\} \right)$ must preserve the empty
relation, its arity must be larger than zero.
If an $n$-ary (for $n > 0$) function $f \in\mathop{\mathrm{Pol}} \left(A,\{\rho,\emptyset\}\right)$ has at least two essential
arguments, for indices~$i$ and~$j$ with $1 \leq i < j \leq n$ say, then there are $(a_1,\dotsc,a_n), (b_1,\dotsc,b_n) \in A^n$
and $a',b' \in A$ such that the values $a,b,c,d\in A$ given by
\begin{align*}
  f(a_1,\dotsc,a_{i-1},a_i,a_{i+1},\dotsc,a_{j-1},a_j,a_{j+1},\dotsc,a_n) &= a\\
  f(a_1,\dotsc,a_{i-1},a',a_{i+1},\dotsc,a_{j-1},a_j,a_{j+1},\dotsc,a_n) &= b\\
  f(b_1,\dotsc,b_{i-1},b_i,b_{i+1},\dotsc,b_{j-1},b_j,b_{j+1},\dotsc,b_n) &= c\\
  f(b_1,\dotsc,b_{i-1},b_i,b_{i+1},\dotsc,b_{j-1},b',b_{j+1},\dotsc,b_n) &= d
\end{align*}
satisfy $a\neq b$ and $c \neq d$. This means that $(a,b,c,d) \notin \rho$, which violates the condition that~$f$ preserves
$\rho$, since $(a_i,a',b_i,b_i), (a_j,a_j,b_j,b')\in \rho$. Hence~$f$ has at most one essential position, the $i$th say
($1 \leq i \leq n$). As~$f$ depends at most on its $i$th position, for every $(x_1, \dotsc, x_n) \in A^n$ we have
$f(x_1, \dotsc, x_n) = f(x_i, x_2, \dotsc, x_n) = f(x_i,x_i, \dotsc, x_n) = \dotsm = f(x_i,\dotsc,x_i)$, which means that
$f = g \circ \pi_i^{(n)}$ for the unary function $g = f\circ(\mathrm{id}_A, \dotsc, \mathrm{id}_A)$, and so
$f \in \langle \mathop{\mathrm{Tr}}(A)\rangle$. The reverse inclusion is trivial.

Now as~$E$ is a closed transformation monoid there is a set of finitary relations~$Q$ on~$A$ (e.g.\ all invariant relations
of~$E$) such that $E = \mathop{\mathrm{End}}(A,Q) = \mathop{\mathrm{Tr}}(A)\cap \mathop{\mathrm{Pol}}(A,Q)$. This implies that
$\mathop{\mathrm{Pol}}(A,Q\cup \{\rho,\emptyset\}) = \mathop{\mathrm{Pol}}(A,Q)\cap
\mathop{\mathrm{Pol}}(A,\{\rho,\emptyset\}) = \mathop{\mathrm{Pol}}(A,Q) \cap \langle \mathop{\mathrm{Tr}}(A)\rangle = \langle E\rangle$, from which it
follows that $\langle E\rangle$ is a closed clone. Indeed, the inclusion
$\langle E \rangle \subseteq \mathop{\mathrm{Pol}}(A,Q)\cap\langle \mathop{\mathrm{Tr}}(A)\rangle$ is immediate. Conversely,
if $f \in \mathop{\mathrm{Pol}}(A,Q) \cap \langle \mathop{\mathrm{Tr}}(A)\rangle$, then there is an arity $n>0$,
an index~$i$ such that $1 \leq i \leq n$, and a unary operation $g \in
\mathop{\mathrm{Tr}}(A)$ such that $f = g\circ \pi_i^{(n)}$. It
follows that $f \circ (\mathrm{id}_A, \dotsc, \mathrm{id}_A) = g$, and so $g \in \mathop{\mathrm{Pol}}(A,Q)$ since
$f \in \mathop{\mathrm{Pol}}(A,Q)$. As~$g$ is unary, $g \in \mathop{\mathrm{End}}(A,Q) = E$. It follows that
$f = g \circ \pi_i^{(n)} \in \langle E \rangle$.
\end{proof}

\begin{lemma}\label{6.2} Let $E\subseteq \mathop{\mathrm{Tr}}(\Omega)$ be a transformation monoid on a countable set~$A$ and
$\theta\colon \langle E\rangle \to \mathcal{O}_\Omega$
be a clone homomorphism from
$\langle E\rangle$ into the clone of all
operations on a countable set~$\Omega$. If the restriction of~$\theta$ to
its unary part
$\theta\restriction_{E}\colon E \to \mathop{\mathrm{Tr}}(\Omega)$ is continuous,
then~$\theta$ is continuous.
\end{lemma}
\begin{proof} In the proof we distinguish two cases. Let $(g_n)_{n\in \mathbb{N}}$ be a sequence of $k$-ary operations of
$\langle E \rangle$ that converges to $g \in
\langle E\rangle^{(k)}$ say. We want to prove
that $\lim_{n\to\infty} \theta(g_n) = \theta(g)$.
Since
\[\langle E \rangle = \bigcup_{k\in \mathbb{N}\setminus \{0\}}\{f\circ \pi_j^{(k)}\colon j\in \{1,\ldots,k\}
\land f\in E\},\]
we have $g_n=f_n \circ \pi_{j_n}^{(k)}$ for all $n \in \mathbb{N}$ with $f_n \in E$ and $1\leq j_n \leq k$ (the index~$j_n$ may
not be uniquely determined in the case that~$f_n$ is constant, but then we make an arbitrary choice, for instance $j_n = 1$),
and $g = f \circ \pi_j^{(k)}$ for some $f \in E$ and~$j$ such that $1 \leq j \leq k$. Let us first note that
\begin{align*}
\lim_{n\to\infty} f_n
&= \lim_{n\to\infty}(g_n\circ(\mathrm{id}_A,\dotsc,\mathrm{id}_A))
= \left(\lim_{n\to\infty} g_n\right) \circ (\mathrm{id}_A,\dotsc,\mathrm{id}_A)
= g\circ (\mathrm{id}_A,\dotsc,\mathrm{id}_A)\\
&= \left(f\circ \pi_j^{(k)}\right)\circ (\mathrm{id}_A,\dotsc,\mathrm{id}_A)
= f\circ \mathrm{id}_A
= f
\end{align*}
since composition of functions is continuous with regard to the product topology.

The collection $\{\{n\in \mathbb{N}\colon j_n = t\}\colon 1\leq t \leq k\}$ consists of disjoint subsets of $\mathbb N$ whose union
covers $\mathbb N$. Since this collection is finite, the set $\{n \in \mathbb{N}\colon j_n = t\}$ must be infinite for at least one
$t\in\{1,\dotsc,k\}$. By shifting the index of the sequence $(g_n)_{n \in \mathbb{N}}$ past the largest element of
the finite members of this collection, we may assume that each of these sets is either infinite or empty. Let
$t_1, \dots, t_\ell$ be the distinct indices in $\{1,\dotsc,k\}$ for which
$I_\nu = \{n\in\mathbb{N}\colon j_n = t_\nu\}$ $(1\leq \nu\leq \ell)$ is infinite. By enumerating $I_\nu$ in strictly increasing
order we get subsequences $(n_{\nu,i})_{i\in\mathbb{N}}\in \mathbb{N}^\mathbb{N}$ such that
$I_\nu = \{n_{\nu,i}\colon i \in \mathbb{N}\}$ and thus $j_{n_{\nu,i}} = t_\nu$ is constant for all $i\in\mathbb{N}$.
The first case we consider is that~$g$, equivalently~$f$, is a constant map. Then $f\circ \pi^{(k)}_j = f\circ \pi^{(k)}_{t_{\nu}}$
holds for every $\nu\in\{1,\dotsc,\ell\}$, and so we can infer that
\begin{align*}
\lim_{i\to\infty} \theta(g_{n_{\nu,i}})
&= \lim_{i\to\infty}
   \theta\left(f_{n_{\nu,i}} \circ \pi^{(k)}_{j_{n_{\nu,i}}}\right)
 = \lim_{i\to\infty}
   \theta\left(f_{n_{\nu,i}} \circ \pi^{(k)}_{t_{\nu}}\right)
 = \lim_{i\to\infty}
   \left(\theta(f_{n_{\nu,i}}) \circ
          \theta\left(\pi^{(k)}_{t_{\nu}}\right)\right)\\
&= \left(\lim_{i\to\infty} \theta(f_{n_{\nu,i}})\right)\circ
   \left(\lim_{i\to\infty} \theta\left(\pi^{(k)}_{t_{\nu}}\right)\right)
 \stackrel{\dagger}{=}
   \theta\left(\lim_{i\to\infty} f_{n_{\nu,i}}\right)\circ
   \theta\left(\pi^{(k)}_{t_{\nu}}\right)\\
&= \theta\left(\lim_{n\to\infty} f_{n}\right)\circ
   \theta\left(\pi^{(k)}_{t_{\nu}}\right)
 = \theta(f)\circ \theta\left(\pi^{(k)}_{t_{\nu}}\right)
 = \theta\left(f\circ \pi^{(k)}_{t_{\nu}}\right)
 = \theta\left(f\circ \pi^{(k)}_{j}\right)
 = \theta(g)
\end{align*}
for each~$\nu$ such that $1\leq \nu\leq\ell$ (the equation marked by~$\dagger$ follows from the continuity of~$\theta$ for unary
operations). Now we have a partition of a sequence into a finite number of subsequences each of which converges to the same
limit~$\theta(g)$. It follows that $\lim_{n\to\infty}\theta(g_n) = \theta(g)$: for every $\varepsilon > 0$ (and~$\nu$ such
that $1 \leq \nu \leq \ell$) we can find  an index $N_\nu\in\mathbb{N}$ such that $\theta\left(g_{n_{\nu,i}}\right)$ has
distance less than $\varepsilon$ from $\theta(g)$ for all $i> N_{\nu}$. Let $N = \max\{n_{\nu,N_{\nu}}\colon 1\leq \nu\leq \ell\}$
and consider $n>N$. Using the partition, we find some
$\nu\in\{1,\dotsc,\ell\}$ and some $i\in\mathbb{N}$ such that
$n = n_{\nu,i}$. If $i \leq N_{\nu}$, then $n = n_{\nu,i}\leq n_{\nu,N_{\nu}} \leq N$ contradicts $n>N$, so $i>N_{\nu}$. This
means that the distance from  $\theta(g_{n}) = \theta(g_{n_{\nu,i}})$ to $\theta(g)$ is less than $\varepsilon$.

The second case of the proof is when~$f$ is not constant. We show that $\ell = 1$ and $t_1 = j$. In order to obtain a
contradiction, let us assume that there is $\nu\in\{1,\dotsc,\ell\}$
where $t = t_\nu \neq j$. No generality is
lost in assuming that $t < j$. Since~$f$ is not constant, there are arguments $x, y$ such that
$f \circ \pi^{(k)}_t (x,\dotsc,x,y,\dotsc,y) = f(x) \neq f(y) = f\circ \pi^{(k)}_{j}(x,\dotsc,x,y,\dotsc,y)$, and the last~$x$ occurs
in the $t$th position. Thus $f \circ \pi^{(k)}_t\neq f\circ \pi^{(k)}_j$, and so
$\varepsilon = d\left(f \circ \pi^{(k)}_t,f\circ \pi^{(k)}_j\right) > 0$. The subsequence $\left(f_{n_{\nu,i}}\right)_{i\in\mathbb{N}}$ converges
to~$f$, and as composition of functions is continuous, the same holds for the sequence
$\left(f_{n_{\nu,i}}\circ \pi^{(k)}_{j_{n_{\nu,i}}}\right)_{i\in\mathbb{N}}
  =\left(f_{n_{\nu,i}}\circ \pi^{(k)}_{t_\nu}\right)_{i\in\mathbb{N}}
  =\left(f_{n_{\nu,i}}\circ \pi^{(k)}_{t}\right)_{i\in\mathbb{N}}$ and
  $f\circ \pi^{(k)}_t$. Let us choose $i\in\mathbb{N}$ large enough that
  $d\left(f_{n_{\nu,i}}\circ \pi^{(k)}_t,f\circ\pi^{(k)}_t \right)
     <\frac{\varepsilon}{2}$. By the triangle inequality,
  \begin{align*}
  \varepsilon =
    d\left(f\circ \pi^{(k)}_j,f\circ \pi^{(k)}_t\right)&\leq
    d\left(f\circ \pi^{(k)}_j,f_{n_{\nu,i}}\circ \pi^{(k)}_t\right)+
    d\left(f_{n_{\nu,i}}\circ \pi^{(k)}_t,f\circ \pi^{(k)}_t\right)\\
    &<d\left(f\circ \pi^{(k)}_j,f_{n_{\nu,i}}\circ \pi^{(k)}_t\right)+
    \frac{\varepsilon}{2},
 \end{align*}
  i.e.\
$  d\left(g,g_{n_{\nu,i}}\right)=
  d\left(f\circ \pi^{(k)}_j, f_{n_{\nu,i}}\circ \pi^{(k)}_{j_{n_{\nu,i}}}\right)=
  d\left(f\circ \pi^{(k)}_j, f_{n_{\nu,i}}\circ \pi^{(k)}_{t_{\nu}}\right)=
  d\left(f\circ \pi^{(k)}_j, f_{n_{\nu,i}}\circ \pi^{(k)}_t\right)
           >\frac{\varepsilon}{2}$
for all sufficiently large $i \in \mathbb{N}$. This means that $(g_n)_{n\in\mathbb{N}}$ cannot converge to~$g$, contrary to our
  overall assumption. Therefore, all (distinct) $t_\nu$ have to be
  equal to~$j$, and thus there can only be one such $t_\nu = t_1 = j$.
  This means that for all but finitely many $n\in\mathbb{N}$ (which we safely
  ignored above) we have $j_n = j$ and thus
  $g_n = f_n \circ  \pi^{(k)}_{j_n} = f_n\circ \pi^{(k)}_j$ for almost all
  $n\in\mathbb{N}$.
  This enables us to conclude that
  \begin{align*}
  \lim_{n\to\infty} \theta(g_n) &=
  \lim_{n\to\infty} \theta\left(f_n\circ\pi^{(k)}_j\right)=
  \lim_{n\to\infty} \left(\theta (f_n)\circ\theta\left(\pi^{(k)}_j\right)\right)
  \stackrel{\ddagger}{=}
   \theta\left(\lim_{n\to\infty} f_n\right) \circ\theta\left(\pi^{(k)}_j\right)\\
  &=\theta(f)\circ\theta\left(\pi^{(k)}_j\right)
  =\theta\left(f\circ \pi^{(k)}_j\right)
  =\theta(g),
  \end{align*}
  where the equation marked by~$\ddagger$ holds because of the assumed continuity of~$\theta$ restricted to unary operations and the continuity of the composition of operations.
\end{proof}

\begin{corollary}\label{6.3} Let $\mathbb{A}$ be a countable relational structure and
$\theta\colon \langle \mathop{\mathrm{End}}(\mathbb{A})\rangle \to C$ be a clone isomorphism between
$\langle \mathop{\mathrm{End}}(\mathbb{A})\rangle$ and a closed clone~$C$ over a countable set\/~$\Omega$. If the restriction of
$\theta$ to its unary part $\theta\restriction_{\mathop{\mathrm{End}}(\mathbb{A})}\colon\mathop{\mathrm{End}}(\mathbb{A}) \to C^{(1)}$ is continuous,
then~$\theta$ is continuous.
\end{corollary}
\begin{proof}
Let us denote by $\iota\colon C\to \mathcal{O}_{\Omega}$ and $\iota'\colon C^{(1)}\to
\mathop{\mathrm{Tr}}(\Omega)$ the inclusion homomorphisms of~$C$
and~$C^{(1)}$ into the full clone and the full transformation monoid on~$\Omega$, respectively. By definition of the subspace
topology on~$C^{(1)}$, $\iota'$ is continuous, so $\iota\theta$ is a clone homomorphism from an essentially at most
unary clone on a countable set into the clone of all operations on the countable carrier set~$\Omega$, whose restriction to
the unary part is $\iota'\theta\restriction_{\mathop{\mathrm{End}}(\mathbb{A})}$ and hence continuous. Letting
$E=\mathop{\mathrm{End}}(\mathbb{A})$ in Lemma~\ref{6.2} we deduce that $\iota\theta$ is continuous; since
$\mathop{\mathrm{im}}(\theta)\subseteq C$, it follows that~$\theta$ is continuous, too.
\end{proof}

As another consequence of Lemma~\ref{6.2}, automatic continuity can be lifted from closed transformation monoids to their
generated clones.
\begin{corollary} \label{6.4} If~$A$ is a countable set, and $E \subseteq \mathop{\mathrm{Tr}}(A)$ is a closed transformation monoid
with automatic continuity, then the essentially at most unary clone~$\langle E\rangle$ generated by it inherits this property.
\end{corollary}
\begin{proof}
By Lemma~\ref{6.1}, the clone~$\langle E\rangle$ is closed. If $\theta\colon \langle E\rangle\to \mathcal{O}_\Omega$ is a clone homomorphism
into the full clone on a countable set~$\Omega$, then its restriction to the unary part is the monoid homomorphism
$\theta\restriction_{E}\colon E\to \mathop{\mathrm{Tr}}(\Omega)$, which is continuous by the assumption on~$E$. By Lemma~\ref{6.2},
$\theta$ is continuous.
\end{proof}

\begin{lemma} \label{6.5} If~$A$ is a countable set, and $E \subseteq \mathop{\mathrm{Tr}}(A)$ is a closed transformation monoid which
has automatic homeomorphicity and its group of invertible members~$G$ acts transitively on~$A$, then~$\langle E \rangle$ also
has automatic homeomorphicity.
\end{lemma}
\begin{proof}
Let $\theta\colon \langle E \rangle\to C$ be a clone isomorphism between $\langle E \rangle$ and another closed clone~$C$ on a
countable set~$\Omega$. Since~$E$ has automatic homeomorphicity and the unary part of~$C$ is closed---because
$C^{(1)} = C \cap \mathop{\mathrm{Tr}}(\Omega)$ and both sets are closed---the restriction $\theta\restriction_E\colon E\to C^{(1)}$
is a homeomorphism. By Corollary~\ref{6.3} we conclude that~$\theta$ is continuous. To see that it must be open too, we use
Proposition~32 from~\cite{Bodirsky}, which holds for clone isomorphisms and is applicable here since~$G$ acts transitively
on~$A$ and~$\theta\restriction_E$ is open.
\end{proof}
From the previous lemma and Theorems~\ref{2.6} and~\ref{4.4} we obtain the result mentioned in the introduction.
\begin{corollary}\label{6.6}
$\langle\mathop{\mathrm{End}}\left(\mathbb{Q},<\right)\rangle$ and
$\langle\mathop{\mathrm{End}}\left(\mathbb{Q},\leq\right)\rangle$
have automatic homeomorphicity.
\end{corollary}

\end{document}